\documentclass[12pt]{article}

\usepackage[english]{babel}
\usepackage{amsmath}
\usepackage{graphicx}
\usepackage{amsfonts}
\usepackage{amssymb}
\usepackage{latexsym}
\newcommand{\bck}{\backslash}

\newenvironment{proof}{\noindent{\bf Proof.}}{$~\hfill\Box$\medbreak}

\DeclareMathOperator{\con}{conv}

\newtheorem{theorem}{Theorem}[section]

\newtheorem{corollary}[theorem]{Corollary}
\newtheorem{definition}[theorem]{Definition}
\newtheorem{lemma}[theorem]{Lemma}

\begin{document}

\begin{center}
{\huge \bf Generalized vector equilibrium problems with pairs of bifunctions and some applications}
 \center{\it by}
 
\center{\it Bui The Hung}
\center{\it Department of Mathematics, Thai Nguyen University of Education\\ e-mail: hungbt@tnue.edu.vn}\\

\end{center}

\date{}

\bigskip

{\bf Abstract.} In this paper, we deal with the following generalized vector equilibrium problem: Let $X, Y$ be topological vector spaces over reals, $D$ be a nonempty subset of $X$, $K$ be a nonempty set and $\theta$ be origin of $Y$. Given multi-valued mapping $F: D\times K\rightrightarrows Y$, can be formulated as the problem, find $\bar x\in D$ such that 
$$\mbox{GVEP}(F, D, K)\,\,\,\,\,\,\theta\in F(\bar x, y)\ \mbox{for all}\ y\in K.$$
We prove several existence theorems for solutions to the generalized vector equilibrium problem when $K$ is an arbitrary nonempty set without any algebraic or topological structure. Furthermore, we establish that some sufficient conditions ensuring the existence of a solution for the considered conditions are imposed not on the entire domain of the bifunctions but rather on a self-segment-dense subset. We apply the obtained results to variational relation problems, vector equilibrium problems, and common fixed point problems.

\bigskip


\bigskip


\section{Introduction and preliminaries}
\hspace{0.5cm} The generalized vector equilibrium problem (for short, ${\rm GVEP}(F, D, K)$) was first introduced and studied by Robinson \cite{Robinson} in the setting of finite dimensional Euclidean spaces. Since then, various kinds of generalized vector equilibrium problems have been intensively studied by many authors in finite and infinite dimensional spaces (see, for example, \cite{D}, \cite{BTH1}, \cite{Lin1}, \cite{Lin2}, \cite{Lin3}, \cite{TH}, \cite{T}, \cite{THM} and the references therein). It is well
known that the generalized vector equilibrium problems provide a unified model of several classes of problems, for example, vector
equilibrium problems, variational relation problems, vector optimization problems and vector saddle point problems. 

The above generalized vector equilibrium problems ${\rm GVEP}(F, D, K)$, in which $D, K$ are constraint sets, $F$ are utility multivalued mappings that are often determined by equalities and inequalities of functions or by inclusions and intersections of other multivalued mappings or by general relations in product spaces. The typical instances
of generalized equilibrium are the following: 

(1) Scalar equilibrium problem: Let $D$ be a nonempty subset in a real topological vector space $X$, $K$ be a nonempty set and $f: D\times K \to \mathbb R$ be a function. If we define the mapping $F: D\times K\rightrightarrows {\mathbb R}$ by
$$F(x, y)=\{t\in \mathbb R: t+f(x, y)\geq 0\},$$
then, problem ${\rm GVEP}(F, D, K)$ is formulated as follows: Find $\bar x \in D \ \mbox{such that}$
$$f(\bar x, y)\geq 0, \ \mbox{for all} \ \ y \in K.$$
 This is called a scalar equilibrium problem. This problem was introduced by Muu and Oettli \cite{MO} in 1992. It is well-known that scalar equilibrium problems are generalizations of variational inequalities and optimization problems, including also optimization-related problems such as fixed point problems, complementarity problems, Nash equilibrium, minimax problems, etc. 

(2) Vector equilibrium problem: Let $C$ be a cone in $Y$. Now, given mapping $f: D\times K\to Y$. We define mapping $F: D\times K\rightrightarrows Y$ as follows:
$$F(x, y)=\{t\in Y: t+f(x, y)\in \Gamma\},$$
where $\Gamma=C$ or $\Gamma=Y\bck (-\operatorname{int} C)$. Then, ${\rm GVEP}(F, D, K)$ becomes: Find $\bar x\in D$ such that 
$$f(\bar x, y)\in \Gamma\ \mbox{for all}\ y\in K.$$
This is called a vector equilibrium problem. Vector quasi-equilibrium problems studied in (\cite{AOS1}, \cite{BLaj}, \cite{BL3}, \cite{FLB}, \cite{BTH3}, \cite{BTH4}, \cite{SMV}, \cite{KhaKr}, \cite{LaV2}, \cite{La1}, \cite{LYG}, \cite{LT}, \cite{Oettli3}, \cite{Oettli2}) which generalizes of scalar equilibrium problem.

(3) Variational relation problem: Let $D$ be a nonempty convex subset of a Hausdorff topological vector space, $K$ be a nonempty set. Let $R$ be a relation on $D\times K$ ( this means that $R$ be a nonempty subset of the product $D\times K$). We define mapping $F: D\times K\rightrightarrows X$ as follows:
$$F(x, y)=\{t\in X: (t+x, y)\in R\}.$$
Then, ${\rm GVEP}(F, D, K)$ becomes: Find $\bar x\in D$ such that 
$$(\bar x, y)\in R\ \mbox{for all}\ y\in K.$$
This is called a variational relation problem. As we all know, variational relation problems were introduced by Luc in \cite{Luc1} as a general model for a large class of problems in nonlinear analysis and applied mathematics, including optimization problems, variational inequalities, variational inclusions, equilibrium problems, etc. Since this manner of approach provides unified results for several mathematical problems, it has been used in many recent papers (see \cite{ABO}, \cite{JHH}, \cite{BL1}, \cite{BL2}, \cite{BLuc}, \cite{DLuc}, \cite{KLuc}, \cite{Lin11}, \cite{LBY}, \cite{LiWa}, \cite{LSS}, \cite{PuY1}, \cite{PuY2}). 

(4) Common fixed point problem: Let $D$ be a nonempty convex subset of a Hausdorff topological vector space, $K$ be a nonempty set. Let $T: D\times K\rightrightarrows D$ be a bifunction. If we define mapping $F: D\times K\rightrightarrows X$ as follows:
$$F(x, y)=x-T(x, y),$$
then, ${\rm GVEP}(F, D, K)$ becomes: Find $\bar x\in D$ such that 
$$\bar x\in T(\bar x, y)\ \mbox{for all}\ y\in K.$$
This is called a common fixed point problem for a family of multivalued maps. As we all know, common fixed point problems were introduced by Balaj (\cite{ABRe}, \cite{BLa4}, \cite{LiChYu}) and obtained some critical results on equilibrium problems and minimax inequalities.

 In order to establish the existence results of equilibrium problems, a usual assumption is monotonicity \cite{BO1}. Several authors have relaxed this rather restrictive assumption to certain types of generalized monotonicity. While most studies assume
pseudomonotonicity, existence results under the more general assumption of quasimonotonicity have also been obtained, using different proof techniques (\cite{BO}, \cite{Bris}, \cite{Ka-scha} and bibliography therein).
Very recently, Oettli \cite{Oettli3} and Chadli et al. \cite{Cha-Ri} proposed an approach to establish the
existence of solutions of equilibrium problems in the vector case by using a
result in the scalar case which involves two bifunctions, rather than one pseudomonotonicity is assumed in this work. The present paper fits into this group of articles. Its goal is to study problem ${\rm GVEP}(F, D, K)$ when the pair of bifunctions $F, G$ satisfy the following condition:
		$$
	\begin{aligned}
		& \text { for each } y \in X \text {, there exists } z \in X \text { such that } \\
		& x \in X, \theta \in G(z, x) \Longrightarrow \theta \in F(x, y),
	\end{aligned}
	$$
	where $G: D\times D\rightrightarrows Y$ is set-valued map. Let $A$ be a subset of topological vector space, we denote by $\operatorname{int} A, \operatorname{cl} A, \con A$ mean the interior, closure, convex hull of $A$, respectively. If $A\subseteq B$, we denote by $\operatorname{cl}_B A$ the closure of $A$ with respect to $B$.  	
\begin{definition}{\rm(\cite{JHH})}\rm{ Let $D$ be a nonempty subsets of a linear space $X$. Then a multivalued map $F: D\rightrightarrows X$ is called KKM mapping if for any $\{x_1,...,x_n\}\subset D$, we have
		$$\con\{x_1,...,x_n\}\subseteq \bigcup_{i=1}^nF(x_i).$$
	}
\end{definition}

\begin{definition}{\rm(\cite{JHH})}\rm{ Let $K$ be nonempty set and $D$ be nonempty convex subsets of a linear space $X$. If $F: K\rightrightarrows D$ satisfies that for any $\{x_1,...,x_n\}\subset K$, there is $\{y_1,...,y_n\}\subset D$ such that
		$$\con\{y_i: i\in J\}\subseteq \bigcup_{i\in J}F(x_i)$$
		for any nonempty subset $J$ of $\{1,...,n\}$, then $F$ is called a generalized KKM mapping.}
\end{definition}

\begin{lemma}{\rm(\cite{JHH})}\label{lem3.3}
	Let $K$ be nonempty set and $D$ be nonempty convex subsets of a topological vector space $X$. If $F: K\rightrightarrows D$ is a generalized KKM mapping with closed values such that $F(x')$ is compact for at least one $x'\in K$, then $\bigcap_{y\in K}F(y)\not=\emptyset$.
\end{lemma}

\section{Main results}
\begin{definition}\rm{ Let $D$ be nonempty subsets of a vector space $X$. If $G: D\times D\rightrightarrows Y$ satisfies that for any $\{x_1, x_2, ...,x_n\}\subset D$ and for any $x\in \con\{x_1,, x_2, ...,x_n\}\cap D$, there is $j\in \{1, 2, ..., n\}$ such that $\theta\in G(x_j, x),$ then $G$ is called a KKM-type mapping on $D$.	}
\end{definition}

\noindent{\bf Remark 1.} Let $D$ be nonempty subsets of a vector space $X$ and $P: D \rightrightarrows X$ be a set-valued mapping. We define the set-valued mapping $G: D \times D \rightrightarrows X$ by
$$G(x, z)=x-P(z)\ \mbox{for all}\ (x, z)\in D \times D.$$
Then, $P$ is a KKM mapping on $D$ if and only if $G$ a KKM-type mapping on $D$.
 \begin{theorem}\label{theo4.0} Assume that $D$ is a nonempty convex compact subset of a vector topological space $X$, $Y$ be a topological vector space and $K$ be a nonempty set. Let $F: D \times K \rightrightarrows Y, G: D \times D \rightrightarrows Y$ be two set-valued mappings satisfying the following conditions:
 	
 	{\rm(i)} for each $y \in K$, the set $\{x \in D: \theta\in F(x, y)\}$ is closed in $D$;
 	
 	{\rm(ii)} $G$ is a KKM-type mapping on $D$;
 	
 	{\rm(iii)} for each $y \in K$, there exists $z \in D$ such that
 	$$x \in D, \theta \in G(z, x) \Longrightarrow \theta \in F(x, y).$$
 		Then, there exists a solution of problem ${\rm GVEP}(F, D, K)$.
 \end{theorem}
 \begin{proof}  Let us define the set-valued mapping $H: K\rightrightarrows D$ by
 	$$H(y)=\{x \in D: \theta\in F(x, y)\}\ \mbox{for all}\ y\in K.$$	
 	We want to show that $H$ is a generalized KKM map. Indeed, for each finite subset $\{y_1, y_2, ..., y_n\}$ of $K$, by condition {\rm(iii)}, there exists $\{z_1, z_2, ..., z_n\}$ be a finite subset of $D$ such that $N(z_i)\subseteq H(y_i)$ for all $i\in \{1, 2, ..., n\}$, where $N: D\rightrightarrows D$ be a set-valued mapping by
 	$$N(z)=\{x \in D: \theta \in G(z, x)\}\ \mbox{for all}\ z\in D.$$
 	 Let $J$ be a nonempty subset of $\{1, 2, ..., n\}$ and let $z\in \con\{z_i: i\in J\}$. By {\rm(ii)}, there is $j\in J$ such that 
 	$$\theta\in G(z_j,z).$$
 	This implies that $z\in N(z_j)\subseteq H(y_j)$ and hence 
 	$$\con\{z_i: i\in J\}\subseteq \bigcup_{i\in J}H(y_i)\ \mbox{for each nonempty subset}\ J\ \mbox{of}\ \{1, 2, ..., n\}.$$
 	This means that $H$ is a generalized KKM mapping. Since the values of $H$ are closed (by the condition {\rm(i)}) and the set $D$ is compact, by Lemma \ref{lem3.3} we infer that $\bigcap_{y\in K}H(y)\not=\emptyset$. Hence, there exists $\bar x\in D$ such that
 	$$\theta \in F\left(\bar x, y\right)\ \mbox{for all}\ y \in K.$$
 	 	
 \end{proof}

Next, we provide an existence theorem for the solutions of generalized vector equilibrium problem ${\rm GVEP}(F, D, K)$ when $D$ is not a compact subset. A subset $D$ of a topological space $X$ is said to be compactly closed if, for each compact subset $M$ of $X$, the set $D\cap M$ is closed in $M$.	
\begin{theorem}\label{theo4.1} Assume that $D$ is a nonempty convex subset of a vector topological space $X$, $Y$ be a topological vector space and $K$ be a nonempty set. Let $F: D \times K \rightrightarrows Y, G: D \times D \rightrightarrows Y$ be two set-valued mappings satisfying the following conditions:
	
	{\rm(i)} for each $y \in K$, the set $\{x \in D: \theta\in F(x, y)\}$ is compactly closed;
	
	{\rm(ii)} $G$ is a KKM-type mapping on $D$;
	
	{\rm(iii)} for each $y \in K$, there exists $z \in D$ such that
		$$x \in D, \theta \in G(z, x) \Longrightarrow \theta \in F(x, y);$$
	
	{\rm(iv)} (Coercivity condition) there exists a nonempty compact subset $M$ of $D$ and a finite subset $L$ of $K$ such that for each $x\in D\bck M$, there exists $y\in L$ satisfied $\theta\not\in F(x, y)$.\\
	Then, there exists a solution of problem ${\rm GVEP}(F, D, K)$ on $M$, that is, there exists $\bar x\in M$ such that $\theta \in F\left(\bar x, y\right)$ for all $y \in K$.
\end{theorem}
\begin{proof} Let $\mathcal{F}_L$ be the family of all finite subsets of $K$ that contain the set $L$ and choose arbitrarily an $B \in \mathcal{F}_L$. From {\rm(iii)}, for each $y \in B$, there exists $z_{y} \in D$ such that
\begin{align}\label{CT2}
	\left\{x \in D: \theta \in G(z_{y}, x)\right\} \subseteq \left\{x \in D: \theta \in F(x, y)\right\}.
\end{align}
We set
	$$Q:=\con\left\{z_{y}: y \in B\right\},$$	
	which is compact. Let us define the set-valued mapping $P:\left\{z_{y}: y \in B\right\} \rightrightarrows Q$ by
	$$P(z)=\operatorname{cl}_{Q}\{x \in Q: \theta \in G(z, x)\}.$$
Now, we want to show that $P$ is a KKM map. From {\rm(ii)}, for each finite subset $\{z_{y_1}, z_{y_2}, ..., z_{y_k}\}$ of $\left\{z_{y}: y \in B\right\}$ and for any $x\in \con\{z_{y_1}, z_{y_2}, ..., z_{y_k}\}$, there exists an index $j\in \{1, 2, ..., k\}$ such that $\theta\in G(z_{y_j}, x)$. This mean that $x\in P(z_{y_j})$. Therefore, 
	$$\con\{z_{y_1}, z_{y_2}, ..., z_{y_k}\}\subseteq \bigcup_{i=1}^k P(z_{y_i}).$$  
Hence, $P$ is a KKM map. By Ky Fan's lemma, we get 
	$$ \bigcap_{y \in B} P\left(z_{y}\right)\not=\emptyset.$$
	Combining {\rm(i)} and (\ref{CT2}), we get
	$$
	\begin{aligned}
		\bigcap_{y \in B} P\left(z_{y}\right)&=\bigcap_{y \in B} \operatorname{cl}_{Q}\left[\left\{x \in D: \theta \in G(z_y, x) \right\} \cap Q\right] \\
		&\subseteq \bigcap_{y \in B} \operatorname{cl}_{Q}[\{x \in D: \theta \in F(x, y)\} \cap Q] \\
		& =\bigcap_{y \in B}\{x \in D: \theta \in F(x, y)\} \cap Q\\
		&=\bigcap_{y \in B}\{x \in Q: \theta \in F(x, y)\} .
	\end{aligned}
	$$
	Therefore, we obtain
	$$\bigcap_{y \in B}\{x \in Q: \theta \in F(x, y)\} \neq \emptyset.$$
	From {\rm(iv)}, we get 
	$$\bigcap_{y \in L}\left\{x \in Q: \theta \in F\left(x, y\right)\right\} \subseteq M.$$
	 It implies
	$$\bigcap_{y \in B}\{x \in Q: \theta \in F(x, y)\} \subseteq M.$$
	Therefore, the set
	$$\mathcal{R}(B):=\{x \in M: \theta \in F(x, y) \text { for all } y \in B\}$$
	is nonempty and closed in $M$. On the other hand, from $\mathcal{R}\left(B_{1} \cup B_{2}\right) \subseteq \mathcal{R}\left(B_{1}\right) \cap \mathcal{R}\left(B_{2}\right)$ for all $B_{1}, B_{2}\in \mathcal{F}_L$, we conclude that the family $\{\mathcal{R}(B): B \in \mathcal{F}_L\}$ has the finite intersection property. Since $M$ is compact, we deduce $\bigcap_{B \in \mathcal{F}_L} \mathcal{R}(B) \neq \emptyset$. Therefore, there exists a point $\bar x \in \bigcap_{B \in \mathcal{F}_L} \mathcal{R}(B)$ satisfies the conclusion of the theorem. 
 \end{proof}

\noindent{\bf Remark 2.} The coercivity condition {\rm(iv)} of Theorem \ref{theo4.1} is clearly satisfied, if $D$ is compact. In this case, Theorem \ref{theo4.1} reduces to Theorem \ref{theo4.0}.

\noindent{\bf Example 1.} Let $K$ be a nonempty of the interval $(-\infty, 0)$ that contain the point $-\frac{1}{2}$ and $X = Y=\mathbb R, D = [0,1)$. Consider the mappings $F: D\times K \rightrightarrows Y$ and $G: D\times D \rightrightarrows Y$ by
$$F(x,y)=\left\{\begin{array}{l}
	(-\infty, xy],\ \mbox{if}\ y\in K, y\leq -1,\\
	{[x+y, 2x+y+2]}, \ \mbox{if}\ y\in K, y>-1,\\
\end{array}\right.$$
$$G(x, z) = (-\infty, x-z],$$
for all $(x, z, y)\in D\times D\times K.$

Obviously, the first part of condition {\rm(i)} of Theorem \ref{theo4.1} is satisfied, since the set 
$$\{x\in D: 0 \in F(x, y)\}=\left\{\begin{array}{l}
	\{0\},\ \mbox{if}\ y\in K, y\leq -1,\\
	{[0, -y]}, \ \mbox{if}\ y\in K, y>-1,\\
\end{array}\right.$$
is closed for all $y\in K$. 

We show that $G$ is a KKM- type on $D$. Indeed, for each finite subset $\{x_1, x_2, ..., x_n\}$ of $D$ and for any $x\in \con\{x_1, x_2, ..., x_n\}$, we put $x_j=\max_{1\leq i\leq n} x_i$. Then $x_j-x\geq 0$. This implies that $0 \in G(x_j, x)$. Therefore, $G$ is a KKM-type on $D$. Hence, the condition {\rm(ii)} of Theorem \ref{theo4.1} is satisfied.

Next, for each $y\in K$, there exists $z=0\in D$ such that if $0 \in G(z, x)$, then we have $z-x=-x\geq 0$. Hence, $x=0$. On the other hand, we have
$$F(0, y)=\left\{\begin{array}{l}
	(-\infty,0],\ \mbox{if}\ y\in K, y\leq -1,\\
	{[y, y+2]}, \ \mbox{if}\ y\in K, y>-1.\\
\end{array}\right.$$
This implies $0 \in F(0, y)$ for all $y\in K.$  Therefore, the condition {\rm(iii)} of Theorem \ref{theo4.1} is satisfied.

Moreover, for $M=[0, \frac{1}{2}]$ is a nonempty compact subset of $D$ and a finite subset $L=\{-\frac{1}{2}\}$ of $K$, we have $0 \notin F(x, -\frac{1}{2})=[x-\frac{1}{2}, 2x+\frac{3}{2}]$ for all $x\in D\bck M=(\frac{1}{2}, 1)$. Thus, the condition {\rm(iv)} of Theorem \ref{theo4.1} is satisfied.

Then all conditions of Theorem \ref{theo4.1} are satisfied and $\bar x=0\in M$ is the unique solution of generalized vector equilibrium problem ${\rm GVEP}(F, D, K)$. 
	
\noindent{\bf Remark 3.} As shown in the following example, assumption {\rm(iii)} in Theorem \ref{theo4.0} and Theorem \ref{theo4.1} is essential.

\noindent{\bf Example 2.} Let $X = Y=\mathbb R, D=K = [0,1]$. Consider the mappings $F: D\times K \rightrightarrows Y$ and $G: D\times D \rightrightarrows Y$ by
$$F(x,y)=(-\infty,-1-x-y],$$
$$G(x, z) = (-\infty, x-z],$$
for all $(x, z, y)\in D\times D\times K.$

Obviously, the first part of condition {\rm(i)} of Theorem \ref{theo4.0} is satisfied, since the set 
$$\{x\in D: 0 \in F(x, y)\}=\emptyset$$
is closed, for all $y\in K$. 

We show that $G$ is a KKM- type on $D$. Indeed, for each finite subset $\{x_1, x_2, ..., x_n\}$ of $D$ and for any $x\in \con\{x_1, x_2, ..., x_n\}$, we put $x_j=\max_{1\leq i\leq n} x_i$. Then $x_j-x\geq 0$. Which implies that $0 \in G(x_j, x)$. Therefore, $G$ is a KKM-type on $D$. Hence, the condition {\rm(ii)} of Theorem \ref{theo4.0} is satisfied.

Now, by $0\notin F(x, 0)$ for all $x\in D$, the condition {\rm(iii)} of Theorem \ref{theo4.0} is not satisfied. Moreover,  the generalized vector equilibrium problem ${\rm GVEP}(F, D, K)$ has no solution.

\begin{corollary}\label{cor4.11} Assume that $D$ is a nonempty convex subset of a vector topological space $X$, $Y$ be a topological vector space and $K$ be a nonempty set. Let $F: D \times K \rightrightarrows Y, G: D \times D \rightrightarrows Y$ be two set-valued mappings satisfying the following conditions:
	
	{\rm(i)} for each $y \in K$, the set $\{x \in D: \theta\in F(x, y)\}$ is compactly closed;
	
	{\rm(ii)} $G$ is a KKM-type mapping on $D$;
	
	{\rm(iii)} for each $y \in K$, there exists $z \in D$ such that
	$$x \in D, \theta \in G(z, x) \Longrightarrow \theta \in F(x, y);$$
	
	{\rm(iv)} (Coercivity condition) there exists a nonempty compact subset $M$ of $D$ and a point $y_0\in K$ such that $\theta\not\in F(x, y_0)$ for all $x\in D\bck M$.\\
Then, there exists a solution of problem ${\rm GVEP}(F, D, K)$ on $M$, that is, there exists $\bar x\in M$ such that $\theta \in F\left(\bar x, y\right)$ for all $y \in K$.
\end{corollary}	
\begin{proof} Let $L=\{y_0\}$. Clearly, $L$ is a finite subset of $K$. Hence, we get Corollary \ref{cor4.11} by Theorem \ref{theo4.1}.

\end{proof}

	\begin{corollary}\label{cor4.12} Assume that $D$ is a nonempty convex subset of a vector topological space $X$, $Y$ be a topological vector space. Let $F: D \times D \rightrightarrows Y$ be a set-valued mappings satisfying the following conditions:
		
		{\rm(i)} for each $z \in D$, the set $\{x \in D: \theta\in F(x, z)\}$ is compactly closed;
		
		{\rm(ii)} $F$ is a KKM-type mapping on $D$;
		
		{\rm(iii)} for each $y \in K$, there exists $z \in D$ such that
		$$x \in X, \theta \in F(z, x) \Longrightarrow \theta \in F(x, y);$$
				
		{\rm(iv)} (Coercivity condition) there exists a nonempty compact subset $M$ of $D$ and a finite subset $L$ of $K$ such that for each $x\in D\bck M$, there exists $y\in L$ satisfied $\theta\not\in F(x, y)$.\\
		Then, there exists a solution of problem ${\rm GVEP}(F, D, K)$ on $M$, that is, there exists $\bar x\in M$ such that $\theta \in F\left(\bar x, y\right)$ for all $y \in D$.
	\end{corollary}	
	\begin{proof} This is immediate from Theorem \ref{theo4.1} with $K=D, F=G.$   	
	\end{proof}	
\hspace{0.5cm}In what follows, we obtain existing result of the solution for generalized vector equilibrium problem ${\rm GVEP}(F, D, D)$, in which the conditions that we consider are imposed on a special type of dense set that we call self-segment-dense \cite{LaV1}. The notion of a self-segment-dense set has been successfully used in the context of scalar and set-valued equilibrium problems in \cite{LaV1}, generalized set-valued monotone operators in \cite{LaV2} and vector equilibrium problems in \cite{La1}. 	
 
\begin{definition}{\rm Consider the sets $U\subseteq V\subseteq X$ and assume that $V$ is convex. We say
		that $U$ is self-segment-dense in $V$, iff $U$ is dense in $V$ and for all $x, y\in U$, the set $[x, y]\cap U$ is dense in $[x, y]$, where $[x, y]=\{z\in X: z=tx+(1-t)y, 0\leq t \leq 1\}$}.
\end{definition}
{\bf Remark 4.} Obviously in one dimension the concept of a self-segment-dense set  is equivalent to the concept of a dense set.

\begin{lemma}{\rm(\cite{LaV2})}\label{lem:1} Let $X$ be a Hausdorff locally convex topological vector
	space, let $V\subseteq X$ be a convex set, and let $U\subseteq V$ a self-segment-dense set in $V$. Then,
	for all finite subset $\{u_1, u_2, . . . , u_n\}\subseteq U$, one has
	$$\operatorname{cl}(\con\{u_1, u_2, . . . , u_n\} \cap U) = \con\{u_1, u_2, . . . , u_n\}.$$
\end{lemma}

\begin{theorem}\label{theo4.3} Assume that $D$ is a nonempty convex and compact subset of a Hausdorff locally convex vector topological space $X$, $Y$ be a topological vector space and $N\subseteq D$ be a self-segment-dense set. Let $F, G: D \times D \rightrightarrows Y$ be two set-valued mappings satisfying the following conditions:
	
	{\rm(i)} for each $z \in N$, the set $\{x \in D: \theta\in F(x, z)\}$ is closed in $D$;
	
	{\rm(ii)} for each $x \in D$, the set $\{z \in D\bck N: \theta\not \in F(x, z)\}$ is open in $D$;
	
	{\rm(iii)} $G$ is a KKM-type on $N$;
	
	{\rm(iv)} for each $z \in N$, there exists $z' \in D$ such that
	$$x \in D, \theta \in G(z', x) \Longrightarrow \theta \in F(x, z).$$	
Then, there exists a solution of problem ${\rm GVEP}(F, D, D)$.
\end{theorem}
\begin{proof} We can prove the theorem using two methods:\\	
	{\bf Method 1 (Use Ky Fan's lemma).} We define the set-valued mapping $H: N\rightrightarrows D$ by  
$$H(z):=\{x \in D: \theta \in F(x, z)\}\ \mbox{for each}\ z \in N.$$
Clearly, for each $z \in N$, $H(z)$ is closed subset since {\rm(i)}. Let $\{z_1, z_2, ..., z_n\}$ be a finite subset of $N$. We show that \begin{align}\label{CT2-1}
\con\{z_1, z_2, ..., z_n\}\cap N\subseteq \bigcup_{i=1}^nH(z_i).
\end{align}
Suppose contrary, then there is $z\in \con\{z_1, z_2, ..., z_n\}\cap N$ such that $z\not\in H(z_i)$ for $i=1, 2, ..., n$. This implies that $\theta\not\in F(z, z_i)$ for $i=1, 2, ..., n$. By the condition {\rm(iv)}, we get $\theta\not\in G(z_i, z)$ for $i=1, 2, ..., n$. This is a contradiction with $G$ is KKM-type on $N$. Hence, (\ref{CT2-1}) is prove. From (\ref{CT2-1}) and by the closedness of $\bigcup_{i=1}^nH(z_i)$, we have
$$\operatorname{cl}(\con\{z_1, z_2, ..., z_n\}\cap N)\subseteq \bigcup_{i=1}^nH(z_i).$$
From Lemma \ref{lem:1}, we get	
$$\con\{z_1, z_2, ..., z_n\}\subseteq \bigcup_{i=1}^nH(z_i).$$
Therfore, $H$ is a KKM map. By Ky Fan's lemma, we get
$$\bigcap_{z\in N}H(z)\not=\emptyset.$$
Then there is $\bar x\in D$ such that $\bar x\in H(z)$ for all $z\in N$. This implies 
$\theta \in F(\bar x, z)$ for all $z\in N$. If there is $z_0\in D\bck N$ such that $\theta \notin F(\bar x, z_0)$, then by {\rm(ii)}, there exists a neighborhood $U$ of $z_0$ such that $\theta \notin F(\bar x, z)$ for all $z\in U$. By the denseness of $N$ in $D$, we have $U\cap N\not=\emptyset$. We can choose $u_0\in U\cap N$, then we have $\theta \notin F(\bar x, u_0)$. This is contradicts with $\theta \in F(\bar x, z)$ for all $z\in N$. Thus, $\theta \in F(\bar x, z)$ for all $z\in D$.\\	
{\bf Method 2 (Use the unit partition theorem).} For any $z\in D$, we put 
	$$\Omega_z:=\{x \in D: \theta \notin F(x, z)\}.$$
	Clearly, $\Omega_z$ is an open subset of $D$ for each $z \in N$ since {\rm(i)}. If we assume that the conclusion of the theorem would be false. Then for each $x\in D$, there is $z\in D$ such that 
\begin{align}\label{CT3}
	\theta\not\in F(x, z). 
	\end{align}
	 We show that $\{\Omega_z\}_{z\in N}$ is an open cover of $D$. Indeed, assume there is $x_0\in D$ such that $x_0\not\in \Omega_z$ for all $z\in N$. Then we have $\theta \in F(x_0, z)$ for all $z\in N$. If there is $z_0\in D\bck N$ such that $\theta \notin F(x_0, z_0)$, then by {\rm(ii)}, there exists a neighborhood $U$ of $z_0$ such that $\theta \notin F(x_0, z)$ for all $z\in U$. By the denseness of $N$ in $D$, we have $U\cap N\not=\emptyset$. We can choose $u_0\in U\cap N$, then we have $\theta \notin F(x_0, u_0)$. This is contradicts with $\theta \in F(x_0, z)$ for all $z\in N$. Thus, $\theta \in F(x_0, z)$ for all $z\in D$, which contradicts (\ref{CT3}).  Consequently, $\{\Omega_z\}_{z\in N}$ is an open cover of $D$. Since $D$ is a compact, there exists $z_1, z_2, ..., z_n \in N$ such that $D\subseteq \cup_{i=1}^n \Omega_{z_i}$. Let us consider a continuous partition of unity functions $(\phi_i)_{i=1, 2, ..., n}$ associated with this open covering of $D$ and define the map $h: \con\{z_1, z_2, ..., z_n\}\to \con\{z_1, z_2, ..., z_n\}$ by 
	$$ h(x)=\sum_{i=1}^n \phi_i(x)z_i.$$
	Obviously $h$ is continuous, and $\con\{z_1, z_2, ..., z_n\}$ is a compact and convex subset of
	the finite-dimensional space $\operatorname{spa}\{z_1, z_2, ..., z_n\}$. Hence, by the Brouwer fixed point
	theorem, there exists $z^*\in \con\{z_1, z_2, ..., z_n\}$ such that $h(z^*) = z^*.$
	Now, we put 
	$$I(z^*)=\{i \in \{1, 2, ..., n\}: \phi_i(z^*)>0\}.$$
	 Then we have $I(z^*)\not=\emptyset$ and $z^*\in \con\{z_i: i\in I(z^*)\}$. By $\phi_i(z^*)>0$ for all $i\in I(z^*)$ then $z^*\in \cap_{i\in I(z^*)}\Omega_{z_i}$. Since $\cap_{i\in I(z^*)}\Omega_{z_i}$ is open, we obtain that there exists an open and convex neighborhood $W$ of $z^*$ such that $W\subseteq\cap_{i\in I(z^*)}\Omega_{z_i}$. Since 
	 $$\con\{z_i: i\in I(z^*)\}\cap W\not=\emptyset$$
	  and by Lemma \ref{lem:1}, we have $$\operatorname{cl}(\con\{z_i: i\in I(z^*)\}\cap N)\cap W\not=\emptyset.$$ 
	This implies $\con\{z_i: i\in I(z^*)\}\cap N\cap W\not=\emptyset$. Let $z_0\in \con\{z_i: i\in I(z^*)\}\cap N\cap W$, where $z_0=\sum_{i\in I(z^*)}\lambda_iz_i$, $\lambda_i\geq 0$ for all $i\in I(z^*)$ and $\sum_{i\in I(z^*)}\lambda_i=1$. On the other hand, since $z_0\in W\subseteq\cap_{i\in I(z^*)}\Omega_{z_i}$, then we have $\theta \notin F(z_0, z_i)$ for all $i\in I(z^*)$. By {\rm(iv)}, we get $\theta \notin G(z,z_0)$ for all $z\in N$. In particular, we have $\theta \notin G(z_i,z_0)$ for all $i\in I(z^*)$. This is a contradiction with $G$ is KKM-type on $N$.	 
\end{proof}		

\noindent{\bf Example 3.}  Let $X=\mathbb R$, $D=[0, 1]$ and $N=(0, 1]$. Then $N$ be a self-segment-dense set in $D$. Consider the mappings $F: D\times D \rightrightarrows \mathbb R$ and $G: D\times D \rightrightarrows \mathbb R$ by
$$F(x,z)=\left\{\begin{array}{l}
	[xz  -\frac{1}{3}, +\infty),\ \mbox{if}\ z\in N,\\
	{[0, +\infty)}, \ \mbox{if}\ z\not\in N,\\
\end{array}\right.$$
$$G(x, z) = (-\infty, x-z],$$
for all $(x, z)\in D\times D.$ Obviously, the condition {\rm(i)} of Theorem \ref{theo4.3} is satisfied, since for $z\in N$, the set 
$$\{x\in D: 0 \in F(x, z)\}=\{x\in D: 0 \in [xz  -\frac{1}{3}, +\infty)\}=[0, \frac{1}{3z}]\cap D$$
is closed in $D$. For each $x\in D$, we have
$$\{z\in D\bck N: 0 \not\in F(x, z)\}=\{z\in D\bck N: 0 \not\in [0, +\infty)\}=
	\emptyset,$$
is open. Hence, the condition {\rm(ii)} of Theorem \ref{theo4.3} is satisfied.

We show that $G$ is a KKM- type on $N$. Indeed, for each finite subset $\{x_1, x_2, ..., x_n\}$ of $N$ and for any $x=\sum_{i=1}^n\lambda_i x_i\in \con\{x_1, x_2, ..., x_n\}$, where $ \sum_{i=1}^n\lambda_i=1, \lambda_i\geq 0$ for all $i=1, 2, ..., n$. We put $x_j =\max_{1\leq i\leq n} x_i$. Then, we get $x_j -x\geq 0$. This implies that $0 \in G(x_j, x)$. Therefore, $G$ is a KKM-type on $N$. Hence, the condition {\rm(iii)} of Theorem \ref{theo4.3} is satisfied.

Next, for each $z\in N$, there exists $z'=0\in D$ such that if $0 \in G(z', x)$, then $z'-x=-x\geq 0$. Hence, $x=0$. On the other hand, we have
$$F(0, z)=[-\frac{1}{3}, +\infty).$$
This implies that $0 \in F(0, z)$ for all $z\in N.$  Therefore, the condition {\rm(iv)} of Theorem \ref{theo4.3} is satisfied.\\
Then all conditions of Theorem \ref{theo4.3} are satisfied, and $\bar x=0$ is a solution of generalized vector equilibrium problem ${\rm GVEP}(F, D, D)$. \\
{\bf Remark 5.} In what follows, we show that the assumption that $N$ is self-segment-dense, in the hypotheses of the previous theorem, is essential, and it cannot be replaced by the denseness of $N$.	

\noindent{\bf Example 4.}   Let us consider the Hilbert space of square-summable sequences $l_2$ and let $D=\{x\in l_2: \Vert x \Vert\leq 1\}, N=\{x\in l_2: \Vert x \Vert=1\}$. Then, it is clear that $N$ is dense with respect to the weak
topology in $D$, but $N$ is not self-segment-dense in $D$ (see \cite{La1}, Example 2.4). Consider the mappings $F: D\times D \rightrightarrows \mathbb R$ and $G: D\times D \rightrightarrows \mathbb R$ by
$$F(x,z)=\left\{\begin{array}{l}
	[1-\Vert z\Vert, +\infty),\ \mbox{if}\ z\not\in N,\\
	{[\Vert x\Vert, +\infty)}, \ \mbox{if}\ z\in N,\\
\end{array}\right.$$
$$G(x, z) = (-\infty, \Vert x\Vert -\Vert z \Vert],$$
for all $(x, z)\in D\times D.$ Obviously, the condition {\rm(i)} of Theorem \ref{theo4.3} is satisfied, since for $z\in N$, the set 
$$\{x\in D: 0 \in F(x, z)\}=\{x\in D: 0 \in [\Vert x\Vert, +\infty)\}=\{0\}$$
is closed. For each $x\in D$, we have
$$\{z\in D\bck N: 0 \not\in F(x, z)\}=\{z\in D\bck N: 0 \not\in [1-\Vert z\Vert, +\infty)\}=D\bck N$$
is open in $D$. Hence, the condition {\rm(ii)} of Theorem \ref{theo4.3} is satisfied.

We show that $G$ is a KKM- type on $N$. Indeed, for each finite subset $\{x_1, x_2, ..., x_n\}$ of $N$ and for any $x=\sum_{i=1}^n\lambda_i x_i\in \con\{x_1, x_2, ..., x_n\}$, where $ \sum_{i=1}^n\lambda_i=1, \lambda_i\geq 0$ for all $i=1, 2, ..., n$. We put $\Vert x_j \Vert=\max_{1\leq i\leq n} \Vert x_i\Vert$. Then, we get 
$$\Vert x\Vert=\Vert \sum_{i=1}^n\lambda_i x_i \Vert\leq \sum_{i=1}^n\lambda_i \Vert x_i\Vert \leq \Vert x_j \Vert.$$
Thus, $\Vert x_j \Vert -\Vert x\Vert\geq 0$. This implies that $0 \in G(x_j, x)$. Therefore, $G$ is a KKM-type on $N$. Hence, the condition {\rm(iii)} of Theorem \ref{theo4.3} is satisfied.

Next, for each $z\in N$, there exists $z'=0\in D$ such that if $0 \in G(z', x)$, then $\Vert z'\Vert-\Vert x\Vert=-\Vert x\Vert\geq 0$. Hence, $x=0$. On the other hand, we have
$$F(0, z)=[0, +\infty).$$
This implies that $0 \in F(0, z)$ for all $z\in N.$  Therefore, the condition {\rm(iv)} of Theorem \ref{theo4.3} is satisfied.

On the other hand, for each $x\in D$, we have
$$0\not\in F(x, z)=[1-\Vert z\Vert, +\infty)\ \mbox{for all}\ z\in D\bck N.$$
Hence, the generalized vector equilibrium problem ${\rm GVEP}(F, D, D)$ has no solution.

	\begin{theorem}\label{theo4.4} 	Assume that $D$ is a nonempty convex subset of a Hausdorff locally convex topological vector space $X$, $Y$ be a topological vector space and $N\subseteq D$ be a self-segment-dense set. Let $F, G: D \times D \rightrightarrows Y$ be two set-valued mappings satisfying the following conditions:

		{\rm(i)} for each $z \in N$, the set $\{x \in D: \theta\in F(x, z)\}$ is compactly closed;
		
		{\rm(ii)} for each $x \in D$, the set $\{z \in D\bck N: \theta \not\in F(x, z)\}$ is open in $D$;
		
		{\rm(iii)} $G$ is a KKM-type on $N$;
		
		{\rm(iv)} for each $z \in N$, there exists $z' \in D$ such that
		$$x \in D, \theta \in G(z', x) \Longrightarrow \theta \in F(x, z).$$
		
		{\rm(v)} (Coercivity condition) there exists a nonempty compact subset $M$ of $D$ and a point $z_0\in D$ such that $\theta\not\in F(x, z_0)$ for all $x\in D\bck M$.\\
		Then, there exists a solution of problem ${\rm GVEP}(F, D, D)$ on $M$, that is, there exists $\bar x\in M$ such that $\theta \in F\left(\bar x, z\right)$ for all $z \in D$.	
		\end{theorem}
		\begin{proof}  We first show that $z_0\in N$. Suppose contrary, then $z_0\in D\bck N$. By {\rm(v)}, we get $z_0\in \{z \in D\bck N: \theta\not\in F(x, z)\}$ for each $x\in D\bck M$. Since {\rm(ii)}, for each $x\in D\bck M$, there exists a neighborhood $W$ of $z_0$ such that 
			$$W\subseteq \{z \in D\bck N: \theta\not\in F(x, z)\}.$$
			From the denseness of $N$ in $D$, we have $W\cap N\not=\emptyset$. By $W\cap N\subseteq W$, then we get 
			$$\emptyset\not= W\cap N\subseteq \{z \in D\bck N: \theta\not\in F(x, z)\},$$
			which a contradiction.	 
			
			Let $\mathcal{F}_N$ be the family of all finite subsets of $N$ that contain the point $z_0$. Let $B \in \mathcal{F}_N$ be arbitrarily set. Then $B\subseteq N$. From {\rm(iv)}, for each $y \in B$, there exists $z_{y} \in D$ such that
			\begin{align}\label{CT4}
			\{x \in D: \theta \in G(z_{y}, x)\} \subseteq \left\{x \in D: \theta \in F(x, y)\right\}.
			\end{align}
			We set
			$$Q:=\con\left\{z_{y}: y \in B\right\},$$	
			which is compact. Let us define the set-valued mapping $H:\left\{z_{y}: y \in B\right\} \rightrightarrows Q$ by
			$$H(z)=\{x \in Q: \theta \in G(z, x)\}.$$
			Let $\{z_{y_1}, z_{y_2}, ..., z_{y_n}\}$ be a finite subset of $\left\{z_{y}: y \in B\right\}$. We show that 
			\begin{align}\label{CT4-1}
				\con\{z_{y_1}, z_{y_2}, ..., z_{y_n}\}\cap N\subseteq \bigcup_{i=1}^nH(z_{y_i}).
			\end{align}
			Suppose contrary, then there is $z\in \con\{z_{y_1}, z_{y_2}, ..., z_{y_n}\}\cap N$ such that $$z\not\in H(z_{y_i})\ \mbox{for}\ i=1, 2, ..., n.$$
			This implies that 
			$$\theta \not\in G(z_{y_i},  z)\  \mbox{for}\ i=1, 2, ..., n.$$
			This is a contradiction with $G$ is a KKM-type on $N$. Hence, (\ref{CT4-1}) is prove. From (\ref{CT4-1}), we get
			$$\operatorname{cl}(\con\{z_{y_1}, z_{y_2}, ..., z_{y_n}\}\cap N)\subseteq \operatorname{cl}\Big(\bigcup_{i=1}^nH(z_{y_i})\Big)=\bigcup_{i=1}^n\operatorname{cl}\big(H(z_{y_i})\big).$$
			From Lemma \ref{lem:1}, we conclude	
			$$\con\{z_{y_1}, z_{y_2}, ..., z_{y_n}\}\subseteq \bigcup_{i=1}^n\operatorname{cl}\big(H(z_{y_i})\big).$$
			Therfore, $\operatorname{cl}H$ is a KKM map, where $\operatorname{cl}H$ is the closure map of the map $H$, defined by $\operatorname{cl}H(z)=\operatorname{cl}(H(z))$ for all $z\in \{z_{y}: y \in B\}$. By Ky Fan's lemma, we get
			$$\bigcap_{y\in B}\operatorname{cl}H(z_y)\not=\emptyset.$$
			Combining {\rm(i)} and (\ref{CT4}), we have
			$$
			\begin{aligned}
				\bigcap_{y \in B} \operatorname{cl}H(z_y)&=\bigcap_{y \in B} \operatorname{cl}_{Q}\left[\left\{x \in D: \theta \in G(z_y, x) \right\} \cap Q\right] \\
				&\subseteq \bigcap_{y \in B} \operatorname{cl}_{Q}[\{x \in D: \theta \in F(x, y)\} \cap Q] \\
				& =\bigcap_{y \in B}\{x \in D: \theta \in F(x, y)\} \cap Q\\
				&=\bigcap_{y \in B}\{x \in Q: \theta \in F(x, y)\} .
			\end{aligned}
			$$
			Therefore, we obtain
			$$\bigcap_{y \in B}\{x \in Q: \theta \in F(x, y)\} \neq \emptyset.$$
			From {\rm(v)}, we get 
			$$\left\{x \in Q: \theta \in F\left(x, z_o\right)\right\} \subseteq M.$$
			It implies
			$$\bigcap_{y \in B}\{x \in Q: \theta \in F(x, y)\} \subseteq M.$$
			Therefore, the set
			$$\mathcal{R}(B):=\{x \in M: \theta \in F(x, y) \text { for all } y \in B\}$$
			is nonempty and closed in $M$. On the other hand, from $\mathcal{R}\left(B_{1} \cup B_{2}\right) \subseteq \mathcal{R}\left(B_{1}\right) \cap \mathcal{R}\left(B_{2}\right)$ for all $B_{1}, B_{2}\in \mathcal{F}_L$, we conclude that the family $\{\mathcal{R}(B): B \in \mathcal{F}_L\}$ has the finite intersection property. Since $M$ is compact, we deduce $\bigcap_{B \in \mathcal{F}_L} \mathcal{R}(B) \neq \emptyset$. Therefore, there exists a point $\bar x$ such that $\bar x \in \bigcap_{B \in \mathcal{F}_L} \mathcal{R}(B)$. This implies that $\bar x\in M$ and
				\begin{align}\label{CT4-2}
		\theta \in F(\bar x, z)\ \mbox{for all}\ z\in N.
			\end{align}
	If there is $z_0\in D\bck N$ such that $\theta \notin F(\bar x, z_0)$, then by {\rm(ii)}, there exists a neighborhood $U$ of $z_0$ such that $\theta \notin F(\bar x, z)$ for all $z\in U$. By the denseness of $N$ in $D$, we have $U\cap N\not=\emptyset$. We can choose $u_0\in U\cap N$, then we have $\theta \notin F(\bar x, u_0)$. This is contradicts with (\ref{CT4-2}). Thus, 
	$$\theta \in F(\bar x, z)\ \mbox{for all}\ z\in D.$$
	\end{proof}		
		
\noindent{\bf Example 5.}  Let $X=\mathbb R$, $D=[0, 1)$ and $N=(0, 1)$. Then $N$ be a self-segment-dense set in $D$. Consider the mappings $F: D\times D \rightrightarrows \mathbb R$ and $G: D\times D \rightrightarrows \mathbb R$ by
$$F(x,z)=\left\{\begin{array}{l}
	[x, +\infty),\ \mbox{if}\ z\in N,\\
	{[0, +\infty)}, \ \mbox{if}\ z\not\in N,\\
\end{array}\right.$$
$$G(x, z) = (-\infty, x-z],$$
for all $(x, z)\in D\times D.$ Obviously, the condition {\rm(i)} of Theorem \ref{theo4.4} is satisfied, since for $z\in N$, the set 
$$\{x\in D: 0 \in F(x, z)\}=\{x\in D: 0 \in [x, +\infty)\}=\{0\}$$
is closed. For each $x\in D$, we have
$$\{z\in D\bck N: 0 \not\in F(x, z)\}=\{z\in D\bck N: 0 \not\in [0, +\infty)\}=
\emptyset,$$
is open. Hence, the condition {\rm(ii)} of Theorem \ref{theo4.4} is satisfied.

We show that $G$ is a KKM- type on $N$. Indeed, for each finite subset $\{x_1, x_2, ..., x_n\}$ of $N$ and for any $x=\sum_{i=1}^n\lambda_i x_i\in \con\{x_1, x_2, ..., x_n\}$, where $ \sum_{i=1}^n\lambda_i=1, \lambda_i\geq 0$ for all $i=1, 2, ..., n$. We put $x_j =\max_{1\leq i\leq n} x_i$. Then, we get $x_j -x\geq 0$. This implies $0 \in G(x_j, x)$. Therefore, $G$ is a KKM-type on $N$. Hence, the condition {\rm(iii)} of Theorem \ref{theo4.4} is satisfied.

Next, for each $z\in N$, there exists $z'=0\in D$ such that if $0 \in G(z', x)$, then $z'-x=-x\geq 0$. Hence, $x=0$. On the other hand, we have
$$0 \in F(0, z)=[0, +\infty)\ \mbox{for all}\ z\in N.$$
Therefore, the condition {\rm(iv)} of Theorem \ref{theo4.4} is satisfied.

Moreover, for $M=[0, \frac{1}{2}]$ is a nonempty compact subset of $D$ and $z_0=\frac{1}{2}\in D$, we have $0 \notin F(x, \frac{1}{2})=[x, +\infty)$ for all $x\in D\bck M=(\frac{1}{2}, 1)$. Thus, the condition {\rm(iv)} of Theorem \ref{theo4.4} is satisfied.

Then all conditions of Theorem \ref{theo4.4} are satisfied, and $\bar x=0$ is a solution of generalized vector equilibrium problem ${\rm GVEP}(F, D, D)$.

\section{Some applications}

\subsection{Application to variational relation problems}

\begin{definition}\rm{ Let $D$ be nonempty convex subsets of a linear space $X$ and let $R'$ be a relation on $D\times D$. We say that $R'$ is KKM on $D\times D$ if for any $\{x_1, x_2, ...,x_n\}\subset D$ and for any $x\in \con\{x_1,, x_2, ...,x_n\}$, there is $j\in \{1, 2, ..., n\}$ such that $(x_j, x)\in R'$.	}
\end{definition}
\begin{corollary}\label{cor3.1} Assume that $D$ is a nonempty convex subset of a vector topological space $X$ and $K$ be a nonempty set. Let $R_1$ be a relation on $D\times K$ and let $R_2$ be a relation on $D\times D$ satisfying the following conditions:
	
	{\rm(i)} for each $y \in K$, the set $\{x \in D: (x, y)\in R_1\}$ is compactly closed;
	
	{\rm(ii)} $R_2$ is a KKM on $D\times D$;
	
	{\rm(iii)} for each $y \in K$, there exists $z \in D$ such that
	$$x \in D, (z, x) \in R_2 \Longrightarrow (x, y) \in R_1;$$
	
	{\rm(iv)} (Coercivity condition) there exists a nonempty compact subset $M$ of $D$ and a finite subset $L$ of $K$ such that for each $x\in D\bck M$, there exists $y\in L$ satisfied $(x, y)\not\in R_1$.\\
	Then, there exists $\bar x\in M$ such that $(\bar x, y) \in R_1$ for all $y \in K$.
\end{corollary}
\begin{proof} We define the set-valued mappings $F: D \times K\rightrightarrows X, G: D \times D\rightrightarrows X$ by
	$$F(x, y)=\{t\in X: (t+x, y)\in R_1\}, $$
	$$G(x, z)=\{t\in X: (t+x, z)\in R_2\},\ \mbox{for all}\ (x, z, y)\in D \times D\times K.$$
Hence, all conditions of Theorem \ref{theo4.1} are satisfied. Apply Theorem \ref{theo4.1}, there exists $\bar x\in M$ such that $\theta \in F\left(\bar x, y\right)$ for all $y \in K$.  
This implies
$$(\bar x, y) \in R_1\ \mbox{for all}\ y\in K.$$ 
\end{proof}
\noindent{\bf Example 6.} Let $K$ be a nonempty of the interval $[\frac{4}{3}, +\infty)$ that contain the point $\frac{4}{3}$ and $X = \mathbb R, D = (0,1]$. Consider the relations $R_1$ on $D \times K$ and $R_2$ on $D \times D$ by
$$R_1=\{(x, y)\in D \times K: x^2-3xy+3\leq 0\},$$
$$ R_2=\{(x, z)\in D \times D: x-z\leq 0\},$$
$\ \mbox{for all}\ (x, z, y)\in D\times D\times K.$        
Obviously, the first part of condition {\rm(i)} of Corollary \ref{cor3.1} is satisfied, since the set 
$$\{x\in D: (x, y)\in R_1\}=\{x\in D: x^2-3xy+3\leq 0\}=[\frac{3y-\sqrt{9y^2-12}}{2}, \frac{3y+\sqrt{9y^2-12}}{2}]$$
is closed for all $y\in K$. 

We show that $R_2$ is a KKM. Indeed, for each finite subset $\{x_1, x_2, ..., x_n\}$ of $D$ and for any $x\in \con\{x_1, x_2, ..., x_n\}$, we put $x_j=\min_{1\leq i\leq n} x_i$. Then $x_j-x\leq 0$, i.e, $(x_j, x)\in R_2$. Therefore, $R_2$ is a KKM. Hence, the condition {\rm(ii)} of Corollary \ref{cor3.1} is satisfied.

Next, for each $y\in K$, there exists $z=1$ such that if $(z, x)\in R_2$, i.e, $z-x=1-x\leq 0$ then $x=1$. This implies that $(1, y)\in R_1$, because $1^2-3y+3=4-3y\leq 0$ for all $y\in K.$  Therefore, the condition {\rm(iii)} of Corollary \ref{cor3.1} is satisfied.

Moreover, for $M=[\frac{1}{2}, 1]$ is a nonempty compact subset of $D$ and a finite subset $L=\{\frac{4}{3}\}$ of $K$, we get $x^2-3xy_0+3= x^2-4x+3>0$ for all $x\in D\bck M=(0, \frac{1}{2})$ and for $y_0=\frac{4}{3}\in L$. This implies that $(x, y_0)\not\in R_1$ for all $x\in D\bck M$. Thus, the condition {\rm(iv)} of Corollary \ref{cor3.1} is satisfied.
 
Then all conditions of Corollary \ref{cor3.1} are satisfied, and $\bar x=1$ is the unique solution of variational relation problem.

\begin{corollary}\label{cor3.2} Assume that $D$ is a nonempty convex subset of a Hausdorff locally convex vector topological space $X$, $N\subseteq D$ be a self-segment-dense set. Let $R_1, R_2$ be are relations on $D\times D$ satisfying the following conditions:
	
	{\rm(i)} for each $z \in N$, the set $\{x \in D: (x, z)\in R_1\}$ is compactly closed;
	
	{\rm(ii)} for each $x \in D$, the set $\{z \in D\bck N: (x, z)\not\in R_1\}$ is open in $D$;
	
	{\rm(iii)} $R_2$ is a KKM-type on $N$;
	
	{\rm(iv)} for each $y \in N$, there exists $z \in D$ such that
	$$x \in D, (z, x)\in R_2 \Longrightarrow (x, y) \in R_1.$$
	
	{\rm(v)} (Coercivity condition) there exists a nonempty compact subset $M$ of $D$ and a point $z_0\in D$ such that $(x, z_0)\not\in R_1$ for all $x\in D\bck M$.\\	
	Then, there exists $\bar x\in M$ such that $\left(\bar x, y\right) \in R_1$ for all $y \in D$.
\end{corollary}

\begin{proof} We define the set-valued mappings $F, G: D \times D\rightrightarrows X$ by
	$$F(x, y)=\{t\in X: (t+x, y)\in R_1\}, $$
	$$G(x, z)=\{t\in X: (t+x, z)\in R_2\},\ \mbox{for all}\ (x, z, y)\in D \times D\times D.$$
	Hence, all conditions of Theorem \ref{theo4.4} are satisfied. Apply Theorem \ref{theo4.4}, there exists $\bar x\in M$ such that $\theta\in F\left(\bar x, y\right)$ for all $y \in D$. This implies
	$$(\bar x, y) \in R_1\ \mbox{for all}\ y\in D.$$ 
\end{proof}

\subsection{Application to vector equilibrium problems}

\begin{corollary}\label{cor4.1} Let $D$ be a convex subset of a topological vector space, $K$ be a nonempty set, $C$ be a cone in a topological vector space $Y$. Let $f: D \times K \rightrightarrows Y, g: D \times D \rightrightarrows Y$ be two set-valued mappings satisfying the following conditions:
	
	{\rm(i)} for each $y \in K$, the set $\{x \in D: f(x, y) \subseteq C\}$ is compactly closed;
	
	{\rm(ii)} for any nonempty finite subset $A$ of $D$ and all $z \in\con(A)$, there exists $x \in A$ such that $g(x, z) \subseteq -C$;
	
	{\rm(iii)} for each $y \in K$, there exists $z \in D$ such that
	$$x \in D, g(z, x) \subseteq -C \Longrightarrow f(x, y) \subseteq C;$$
	
	{\rm(iv)} (Coercivity condition) there exists a nonempty compact subset $M$ of $D$ and a finite subset $L$ of $K$ such that for each $x\in D\bck M$, there exists $y\in L$ satisfied $f\left(x, y_{0}\right) \not\subseteq C$.
	
	Then, there exists $\bar x \in M$ such that $f\left(\bar x, y\right) \subseteq C$ for all $y \in K$.
\end{corollary}
\begin{proof} We define the set-valued mappings $F: D \times K\rightrightarrows Y, G: D \times D\rightrightarrows Y$ by
	$$F(x, y)=\{t\in Y: t+f(x, y)\subseteq C\}, $$
	$$G(x, z)=\{t\in Y: t+g(x, y)\subseteq -C\},\ \mbox{for all}\ (x, z, y)\in D \times D\times K.$$
Hence, all conditions of Theorem \ref{theo4.1} are satisfied. Applying Theorem \ref{theo4.1}, there exists $\bar x\in D$ such that 
	$$\theta\in F(\bar x, y)\ \mbox{for all}\ y\in K.$$
	This implies
$$f\left(\bar x, y\right) \subseteq C\ \mbox{for all}\ y\in K.$$
\end{proof}

\begin{corollary}\label{cor4.2} Let $D$ be a convex subset of a topological vector space, $K$ be a nonempty set, $C$ be a cone with nonempty interior in a topological vector space $Y$. Let $f: D \times K \rightrightarrows Y, g: D \times D \rightrightarrows Y$ be two set-valued mappings satisfying the following conditions:
	
	{\rm(i)} for each $y \in K$, the set $\{x \in D: f(x, y) \not\subseteq -\operatorname{int} C\}$ is compactly closed;
	
	{\rm(ii)} for any nonempty finite subset $A$ of $D$ and all $z \in \con(A)$, there exists $x \in A$ such that $g(x, z) \subseteq -C$;
	
	{\rm(iii)} for each $y \in K$, there exists $z \in D$ such that
	$$x \in D, g(z, x) \subseteq -C \Longrightarrow f(x, y) \not\subseteq -\operatorname{int} C;$$
	
	{\rm(iv)} (Coercivity condition) there exists a nonempty compact subset $M$ of $D$ and a finite subset $L$ of $K$ such that for each $x\in D\bck M$, there exists $y\in L$ satisfied $f\left(x, y\right) \subseteq -\operatorname{int} C$.
	
	Then, there exists $\bar x \in M$ such that $f\left(\bar x, y\right) \not\subseteq -\operatorname{int} C$ for all $y \in K$.
\end{corollary}
\begin{proof} We define the set-valued mappings $F: D \times K\rightrightarrows Y, G: D \times D\rightrightarrows Y$ by
	$$F(x, y)=\{t\in Y: t+f(x, y)\not\subseteq -\operatorname{int} C\}, $$
	$$G(x, z)=\{t\in Y: t+g(x, y)\subseteq -C\},\ \mbox{for all}\ (x, z, y)\in D \times D\times K.$$
Hence, all conditions of Theorem \ref{theo4.1} are satisfied. Applying Theorem \ref{theo4.1}, there exists $\bar x\in D$ such that 
	$$\theta\in F(\bar x, y)\ \mbox{for all}\ y\in K.$$
	This implies
	$$f\left(\bar x, y\right) \not\subseteq -\operatorname{int} C\ \mbox{for all}\ y\in K.$$
\end{proof}
\noindent{\bf Remark 6.} Corolarry \ref{cor4.2} is a new result for the existence of solutions to the vector equilibrium problem. Observe that our result is different from those obtained in \cite{BL3}. Namely, the authors consider the vector equilibrium problem with the objective function determined on the Cartesian product of two equal sets, while we think this problem with the objective function determined on the Cartesian product of two different sets, where one set is arbitrarily non-empty. Furthermore, in Theorem 7 of \cite{BL3},  the authors used the assumption that $X$ is Hausdorff topological vector space and $C$ be a closed convex cone in a locally convex Hausdorff topological vector space, while our results use the topological vector spaces and $C$ does not need convexity and closure. Example 7 below shows that sometimes our Corollary \ref{cor4.2} is useful, while Theorem 7 of \cite{BL3} is not.

\noindent{\bf Example 7.} Let $K$ be a nonempty of the interval $(-3, 0)$ that contain the point $-\frac{5}{2}$ and $X = Y=\mathbb R, D = [0,1)$ and $C=\mathbb R_+$. Consider the mappings $f: D\times K \rightrightarrows Y$ and $g: D\times D \rightrightarrows Y$ by
$$f(x,y)=\left\{\begin{array}{l}
	(-\infty, xy],\ \mbox{if}\ y\in K, y< -2,\\
	{[x+y, x+y+2]}, \ \mbox{if}\ y\in K, y\geq-2,\\
\end{array}\right.$$
$$g(x, z) = (-\infty, x-z],$$
for all $(x, z, y)\in D\times D\times K.$

Obviously, the first part of condition {\rm(i)} of Corollary \ref{cor4.2} is satisfied, since the set 
$$\{x\in D: f(x, y)\not\subseteq -\operatorname{int} C\}=\left\{\begin{array}{l}
	\{0\},\ \mbox{if}\ y\in K, y<-2,\\
	{D}, \ \mbox{if}\ y\in K, y\geq-2,\\
\end{array}\right.$$
is compactly closed for all $y\in K$. 

For each finite subset $\{x_1, x_2, ..., x_n\}$ of $D$ and for any $x\in \con\{x_1, x_2, ..., x_n\}$, we put $x_j=\min_{1\leq i\leq n} x_i$. Then $x_j-x\leq 0$. Which implies that $g(x_j, x) \subseteq -C$. Hence, the condition {\rm(ii)} of Corollary \ref{cor4.2} is satisfied.

Next, for each $y\in K$, there exists $z=0\in D$ such that if $g(z, x)\subseteq -C$, then $z-x=-x\leq 0$. Hence, $x=0$. On the other hand, we have
$$f(0, y)=\left\{\begin{array}{l}
	(-\infty, 0],\ \mbox{if}\ y\in K, y<-2,\\
	{[y, y+2]}, \ \mbox{if}\ y\in K, y\geq-2.\\
\end{array}\right.$$
This implies that $f(0, y)\not\subseteq -\operatorname{int} C$, for all $y\in K.$  Therefore, the condition {\rm(iii)} of Corollary \ref{cor4.2} is satisfied.

Moreover, for $M=[0, \frac{1}{2}]$ is a nonempty compact subset of $D$ and a finite subset $L=\{-\frac{5}{2}\}$ of $K$, we have $f(x, -\frac{5}{2})=(-\infty, -\frac{5}{2}x]\subseteq -\operatorname{int} C$ for all $x\in D\bck M=(\frac{1}{2}, 1)$. Thus, the condition {\rm(iv)} of Corollary \ref{cor4.2} is satisfied.

Then all conditions of Corollary \ref{cor4.2} are satisfied, and $\bar x=0$ is a solution of vector equilibrium problem. By $D\not=K$, then Theorem 7 of \cite{BL3} is not applicable.

\begin{corollary}\label{cor4.5} Assume that $D$ is a nonempty convex subset of a Hausdorff locally convex vector topological space $X$, $Y$ be a topological vector space, $C$ be a cone in a topological vector space $Y$, $N\subseteq D$ be a self-segment-dense set. Let $f, g: D \times D \rightrightarrows Y$ be two set-valued mappings satisfying the following conditions:
	
	{\rm(i)} for each $z \in N$, the set $\{x \in D: f(x, z) \subseteq C\}$ is compactly closed;
	
	{\rm(ii)} for each $x \in D$, the set $\{z \in D\bck N: f(x, z) \not\subseteq C\}$ is open in $D$; 
		
	{\rm(iii)} for any nonempty finite subset $A$ of $N$ and all $z \in \con(A)\cap N$, there exists $x \in A$ such that $g(x, z) \subseteq -C$;
	
	{\rm(iv)} for each $z \in N$, there exists $z' \in D$ such that
	$$x \in D, g(z', x) \subseteq -C \Longrightarrow f(x, z) \subseteq C;$$	
	
	{\rm(v)} (Coercivity condition) there exists a nonempty compact subset $M$ of $D$ and a point $z_0\in D$ such that $f(x, z_0)\not\subseteq C$ for all $x\in D\bck M$.\\		
	Then, there exists $\bar x \in M$ such that $f\left(\bar x, x\right) \subseteq C$ for all $x \in D$.
\end{corollary}
\begin{proof} We define the set-valued mappings $F, G: D \times D\rightrightarrows Y$ by
	$$F(x, z)=\{t\in Y: t+f(x, z)\subseteq C\}, $$
	$$G(x, z)=\{t\in Y: t+g(x, z)\subseteq -C\},\ \mbox{for all}\ (x, z)\in D \times D.$$
	Hence, all conditions of Theorem \ref{theo4.4} are satisfied. Applying Theorem \ref{theo4.4}, there exists $\bar x\in M$ such that 
	$$\theta\in F(\bar x, y)\ \mbox{for all}\ y\in D.$$
	This implies
	$$f\left(\bar x, y\right) \subseteq C\ \mbox{for all}\ y\in D.$$
\end{proof}

\begin{corollary}\label{cor4.6} Assume that $D$ is a nonempty convex subset of a Hausdorff locally convex vector topological space $X$, $Y$ be a topological vector space, $C$ be a cone with nonempty interior in a topological vector space $Y$, $N\subseteq D$ be a self-segment-dense set. Let $f, g: D \times D \rightrightarrows Y$ be two set-valued mappings satisfying the following conditions:
	
	{\rm(i)} for each $z \in N$, the set $\{x \in D: f(x, z) \not\subseteq -\operatorname{int} C\}$ is compactly closed;
	
	{\rm(ii)} for each $x \in D$, the set $\{z \in D\bck N: f(x, z) \subseteq -\operatorname{int} C\}$ is open in $D$;
	
	{\rm(iii)} for any nonempty finite subset $A$ of $N$ and all $z \in \con(A)\cap N$, there exists $x \in A$ such that $g(x, z) \subseteq -C$;
	
	{\rm(iv)} for each $z \in N$, there exists $z'\in D$ such that
	$$	x \in D, g(z, x) \subseteq -C \Longrightarrow f(x, y) \not\subseteq -\operatorname{int} C.$$
	
 {\rm(v)} (Coercivity condition) there exists a nonempty compact subset $M$ of $D$ and a point $z_0\in D$ such that $f(x, z_0)\not\subseteq C$ for all $x\in D\bck M$.\\		
Then, there exists $\bar x \in M$ such that $f\left(\bar x, x\right) \not\subseteq -\operatorname{int} C$ for all $x \in D$.
\end{corollary}
\begin{proof} We define the set-valued mappings $F, G: D \times D\rightrightarrows Y$ by
	$$F(x, y)=\{t\in Y: t+f(x, y)\not\subseteq -\operatorname{int} C\}, $$
	$$G(x, z)=\{t\in Y: t+g(x, y)\subseteq -C\},\ \mbox{for all}\ (x, z, y)\in D \times D\times D.$$
	Hence, all conditions of Theorem \ref{theo4.4} are satisfied. Applying Theorem \ref{theo4.4}, there exists $\bar x\in M$ such that 
	$$\theta\in F(\bar x, y)\ \mbox{for all}\ y\in D.$$
	This implies
	$$f\left(\bar x, y\right) \not\subseteq -\operatorname{int} C\ \mbox{for all}\ y\in D.$$
\end{proof}

\subsection{Application to common fixed point theorems}

\begin{corollary}\label{cor5.1} Assume that $D$ is a nonempty convex subset of a vector topological space $X$ and $K$ be a nonempty set. Let $T: D \times K \rightrightarrows X, Q: D \times D \rightrightarrows X$ be two set-valued mappings satisfying the following conditions:
	
	{\rm(i)} for each $y \in K$, the set $\{x \in D: x\in T(x, y)\}$ is compactly closed;
	
	{\rm(ii)} for any $\{x_1, x_2, ...,x_n\}\subset D$ and for any $x\in \con\{x_1,, x_2, ...,x_n\}$, there is an index $j\in \{1, 2, ..., n\}$ such that $x\in Q(x_j, x)$;
	
	{\rm(iii)} for each $y \in K$, there exists $z \in D$ such that
	$$x \in D, z \in Q(z, x) \Longrightarrow x \in T(x, y);$$
	
	{\rm(iv)} there exists a nonempty compact subset $M$ of $D$ and a finite subset $L$ of $K$ such that for each $x\in D\bck M$, there exists $y\in L$ satisfied $x\not\in T(x, y)$.\\
	Then, there exists $\bar x\in M$ such that $\bar x \in T\left(\bar x, y\right)$ for all $y \in K$.
\end{corollary}

\begin{proof} 
We define the set-valued mappings $F: D \times K\rightrightarrows X, G: D \times D\rightrightarrows X$ by
$$F(x, y)=x-T(x, y),G(x, z)=x-Q(x, z),\ \mbox{for all}\ (x, z, y)\in D \times D\times K.$$
Hence, all conditions of Theorem \ref{theo4.1} are satisfied. Applying Theorem \ref{theo4.1}, there exists $\bar x\in D$ such that 
$$\theta\in F(\bar x, y)\ \mbox{for all}\ y\in K.$$
This implies
$$\bar x \in T\left(\bar x, y\right)\ \mbox{for all}\ y\in K.$$	
\end{proof}
\noindent{\bf Remark 7.} Corolarry \ref{cor5.1} is a new result for the existence of solutions to the common fixed point problem. Observe that our result is different from those obtained in (\cite{ABRe}, \cite{BLa4}, \cite{LiChYu}). Namely, in Theorem 1 of \cite{BLa4} and Theorem 3.2 of \cite{LiChYu},  the authors used the assumption that the mapping $T$ is compact, but our results do not need; in Theorem 4.1- Theorem 4.5 of \cite{ABRe}, the authors use the assumption that $D$ is a closed set, while our results do not. Example 8 below shows that sometimes our Corollary \ref{cor5.1} is valid, while Theorem 1 of \cite{BLa4}, Theorem 3.2 of \cite{LiChYu}, Theorem 4.1- Theorem 4.5 of \cite{ABRe} are not.

\noindent{\bf Example 8.} Let $K$ be a nonempty of the interval $(-\infty, 0)$ that contain the point $-1$ and $X = \mathbb R, D = [0, 2)$. Consider the mappings $T: D\times K \rightrightarrows X$ and $Q: D\times D \rightrightarrows X$ by
$$T(x,y)=\left\{\begin{array}{l}
	(-\infty,1-x+xy],\ \mbox{if}\ y\in K, y< -1,\\
	{[2x+y, 3x+y+2]}, \ \mbox{if}\ y\in K, y\geq-1,\\
\end{array}\right.$$
$$Q(x, z) = (-\infty, 2x-z],$$
for all $(x, z, y)\in D\times D\times K.$

The condition {\rm(i)} of Corollary \ref{cor5.1} is satisfied, since 
$$\{x\in D: x\in T(x, y)\}=\left\{\begin{array}{l}
	[0, \frac{1}{2-y}],\ \mbox{if}\ y\in K, y<-1,\\
	{[-\frac{y}{2}-1, -y]}\cap D, \ \mbox{if}\ y\in K, y\geq-1.\\
\end{array}\right.$$

For each finite subset $\{x_1, x_2, ..., x_n\}$ of $D$ and for any $x\in \con\{x_1, x_2, ..., x_n\}$, we put $x_j=\max_{1\leq i\leq n} x_i$. Then $x_j \geq x$. This implies that $x\in Q(x_j, x)$. Hence, the condition {\rm(ii)} of Corollary \ref{cor5.1} is satisfied.

Next, for each $y\in K$, there exists $z=0\in D$ such that if $0\in Q(0, x)$, then $$z-x=-x\geq 0.$$
 Hence, $x=0$. On the other hand, we have
$$T(0,y)=\left\{\begin{array}{l}
	(-\infty,1],\ \mbox{if}\ y\in K, y< -1,\\
	{[y, y+2]}, \ \mbox{if}\ y\in K, y\geq-1.\\
\end{array}\right.$$
This implies that $0\in T(0, y)$ for all $y\in K.$  Therefore, the condition {\rm(iii)} of Corollary \ref{cor5.1} is satisfied.

Moreover, for $M=[0, 1]$ is a nonempty compact subset of $D$ and a finite subset $L=\{-1\}$ of $K$, we have $x\not\in T(x, -1)=[2x-1, 3x+1]$ for all $x\in D\bck M=(1, 2)$. Thus, the condition {\rm(iv)} of Corollary \ref{cor5.1} is satisfied.

Then all conditions of Corollary \ref{cor5.1} are satisfied, and $\bar x=0$ is a common fixed point of family mappings $\{T(., y)\}_{y\in K}$. On the other hand, since $T(D\times K)$ containing the set $(-\infty,0]$, so $\overline{T(D\times K)}$ is not a compact set, hence $T$ is not compact. Therefore, Theorem 1 of \cite{BLa4} and Theorem 3.2 of \cite{LiChYu} are not applicable. Moreover, Theorem 4.1- Theorem 4.5 of \cite{ABRe} are not applicable since $D$ is not closed.

\begin{corollary}\label{cor5.2} Assume that $D$ is a nonempty convex subset of a Hausdorff locally convex vector topological space $X$, $N\subseteq D$ be a self-segment-dense set. Let $T, Q: D \times D \rightrightarrows X$ be two set-valued mappings satisfying the following conditions:
	
	{\rm(i)} for each $y \in N$, the set $\{x \in D: x\in T(x, y)\}$ is compactly closed;
	
	{\rm(ii)} for each $x \in D$, the set $\{y \in D\bck N: x\not\in T(x, y)\}$ is open in $D$;
	
	{\rm(iii)} for any $\{x_1, x_2, ...,x_n\}\subset N$ and for any $x\in \con\{x_1,, x_2, ...,x_n\}\cap N$, there is an index $j\in \{1, 2, ..., n\}$ such that $x\in Q(x_j, x)$;
	
	{\rm(iv)} for each $y \in N$, there exists $z \in D$ such that
	$$x \in D, z \in Q(z, x) \Longrightarrow x \in T(x, y);$$
	
	{\rm(v)} (Coercivity condition) there exists a nonempty compact subset $M$ of $D$ and a point $z_0\in D$ such that $x \not \in T(x, z_0)$ for all $x\in D\bck M$.\\	
	Then, there exists $\bar x\in M$ such that $\bar x \in  T\left(\bar x, y\right)$ for all $y \in D$.

\end{corollary}

\begin{proof} 
	We define the set-valued mappings $F: D \times K\rightrightarrows X, G: D \times D\rightrightarrows X$ by
	$$F(x, y)=x-T(x, y),G(x, z)=x-Q(x, z),\ \mbox{for all}\ (x, z, y)\in D \times D\times K.$$
	Hence, all conditions of Theorem \ref{theo4.4} are satisfied. Applying Theorem \ref{theo4.4}, there exists $\bar x\in D$ such that 
	$$\theta\in F(\bar x, y)\ \mbox{for all}\ y\in K.$$
	This implies
	$$\bar x \in T\left(\bar x, y\right)\ \mbox{for all}\ y\in K.$$	
\end{proof}

\section{Conclusions}
  The existence theorems for ${\rm GVEP}(F, D, K)$ we present can be seen as a foundational contribution to the field of vector equilibrium problems. By establishing the existence of solutions under minimal structural assumptions on $K$ and by focusing on self-segment-dense subsets, we provide tools that can be applied to a wide range of problems in optimization. Overall, the paper advances the theory of vector equilibrium problems by providing new existence results under general conditions and by demonstrating the practical utility of these results through applications to related problem areas.

\end{document}